\documentclass[twoside,reqno]{amsart}
\usepackage{hyperref}
\usepackage{amssymb,latexsym,amstext,amsthm}
\usepackage{verbatim} 
\usepackage{eucal} 
\usepackage{color}
\usepackage{amsmath}

\theoremstyle{plain}
\newtheorem{thm}{Theorem}[section]
\newtheorem{pro}[thm]{Proposition}
\newtheorem{cor}[thm]{Corollary}
\newtheorem{lem}[thm]{Lemma}

\theoremstyle{definition}

\newtheorem{DEF}[thm]{Definition}
\newtheorem{rem}[thm]{Remark}

\def\rd{\dot R}
\def\ad{\dot A}
\def\rt{R^\times}
\def\rtid{\tilde{R}}
\def\w{\mathcal{W}}
\def\wdt{\dot\w}
\def\v{\mathcal{V}}

\def\ba{\mathfrak{B}}

\def\frm{(\cdot,\cdot)}
\def\Aut{\mathrm{Aut}}
\def\sgn{\mathrm{sgn}}
\def\andd{\quad\mathrm{and}\quad}
\def\Alt{\mathrm{Alt}}
\def\Ker{\mathrm{Ker}}

\def\hti{\mathrm{ht}}

\def\a{\alpha}
\def\b{\beta}
\def\g{\gamma}
\def\ta{\tau}
\def\ep{\epsilon}
\def\sg{\sigma}
\def\vep{\varepsilon}

\def\Lam{\Lambda}
\def\d{\delta}

\begin{document}

\setcounter{page}{1} \setcounter{page}{1}
\title[length function for Weyl groups of type $A_1$]{A Length Function for Weyl Groups of extended affine root systems of Type $A_1$}

\author{ Saeid Azam and Mohammad Nikouei}
\thanks{S. Azam: School of Mathematics, Institute for Research in Fundamental Sciences (IPM), P.O.
Box 19395-5746, Tehran, Iran, and Department of Mathematics, University of Isfahan, P.O. Box
81745-163, Isfahan, Iran (Corresponding Author).\\
Email: azam@sci.ui.ac.ir}
\thanks{M. Nikouei: Department of mathematics, University of Isfahan, P.O.Box. 81745-163, Isfahan, Iran.\\
Email:  nikouei@sci.ui.ac.ir, nikouei2002@yahoo.com}



\begin{abstract}
In this work, we study the concept of the ``length function'' and some of its combinatorial properties for the class of extended affine root systems of type $A_1$. We introduce a notion of root basis for these root systems, and using a unique expression of the elements of the Weyl group with respect to a set of generators for the Weyl group, we calculate the length function with respect to a very specific root basis.
\end{abstract}
\maketitle
\setcounter{section}{-1}
\section{\bf Introduction}\label{introduction}

The combinatorial aspects of Weyl groups have always been one of
the most important parts of the Lie theory. Among all such
aspects, the concept of the ``length'' for a Weyl group element
and its various applications to the whole theory has been of
great interest. In this work, we consider the concept of the
``length function'' and some of its combinatorial properties
for the class of extended affine root systems of type $A_1$.
This is the first attempt in this regard, and we hope our
approach offers a model for studying the same concepts for
extended affine root systems of other types.

For finite and affine cases, the concept of length, with respect to the so called root bases, is crucial in order to show that the corresponding Weyl groups have the Coxeter presentation \cite[Chapter 5 and 16]{C}.
One knows that in these two cases, the set of simple reflections generate the Weyl group, and that the length of a Weyl group element $w$ is characterized by the number of positive roots mapped to negative roots, by $w$.
Finite and affine root systems are in fact extended affine root systems of nullities $0$ and $1$, respectively.

In case of extended affine root systems of nullity greater that 1, there is no such notion of root basis (see \cite[Section 5]{AEAWG}). Also in this case extended affine Weyl groups are not Coxeter groups \cite{HCox}.

In Section \ref{ARS}, using the definition of affine reflection systems
from \cite{AYY}, we recall some essential properties of the
corresponding Weyl groups of type $A_1$ from \cite{AN}. An affine
reflection system is a generalization of an extended affine
root system. In Proposition \ref{baby W}, we show that in contrast to
extended affine Weyl groups, there is exactly one Weyl group
associated to all possible affine reflection systems of type
$A_1$, in the same ground abelian group. Let $R$ be an affine
reflection system of type $A_1$ in an abelian group A, with
Weyl group $\w$. Let $A^0$ be the radical of $A$. We introduce
two maps $\vep:\w\rightarrow \{-1,1\}$ and $T:A^0\rightarrow
A^0$, which will be crucial for the rest of the work. Fixing
a finite root system $\rd=\{0,\pm\ep\}$ of type $A_1$, in
Theorem \ref{uniexpspro}, we show that each element $w\in\w$ has a unique
expression in the form
$$
w=w_\epsilon^{\delta_{\vep(w),1}}w_{\epsilon+T(w)}.$$ In
particular, the Weyl group is isomorphic to the semidirect
product of the finite Weyl group of type $A_1$ and the radical
of the form on $A$.


Section \ref{wlf} is the core of this work. In this section we consider the so called toroidal root basis $\rtid$ of type $A_1$. From definition of $\rtid$ in (\ref{TARS}), any extended affine root system can be considered as a subset of $\rtid$. This section has two parts. In the first part, we define a height function for $\rtid$ (Definition \ref{ht}). We also partition $\rtid$, and so any extended affine root system of type $A_1$, into a positive and a negative part (Definition \ref{PNroot}). Then in Lemma \ref{hipn}, we show that a root is positive (resp. negative) if and only if its height is positive (resp. negative). Next, we introduce a notion of root basis (Definition \ref{RootBasis}) and consider the fundamental root basis, the unique root basis whose height of all its elements is one. The second part of this section is dedicated to the calculation of the length of elements of $\w$ with respect to the fundamental root basis, using a combinatorial approach (see Proposition \ref{reflectionL} and Theorem \ref{LFT}). In this part, we also show that for each $w\in\w$, there exists a root $\a$ in $\rtid$ such that the length of $w$ with respect to the fundamental root basis is equal to the absolute value of the height of $\a$ (see Proposition \ref{reflectionL} and Corollary \ref{W0len}).

Section \ref{rbol} is dedicated to a discussion on the $\w$-orbits of root bases. We show that for nullity greater than 1, the number of $\w$-orbits is  not finite. We also calculate the length of each element of $\w$ with respect to any root basis which is a $\w$-conjugate of the fundamental root basis, see Corollary \ref{wPLen}.

As mentioned before, the concept of length for elements of the finite and affine Weyl groups are known and calculated. In Section \ref{akmc}, we show that the classical length function, for an affine Weyl group of type $A_1$, coincides with the length function described in Section \ref{wlf}. We also show that the classical definition of root basis, in this case, is equivalent to Definition \ref{RootBasis}.

\section{\bf Affine reflection systems of type $A_1$ and their Weyl groups}\setcounter{thm}{0}\label{ARS}
The class of affine reflection systems was first introduced by E. Neher and O. Loos in \cite{LNRS}. Here we use an equivalent definition given in \cite{AYY}. Let $A$ be an abelian group and $\frm:A\times A\longrightarrow\mathbb{Q}$ be a symmetric bi-homomorphism on $A$, where $\mathbb{Q}$ is the field of rational numbers. $\frm$ is called a form on $A$. The subgroup
$$A^0:=\{\a\in A\;|\;(\a,A)=0\}$$
of $A$ is called the radical of the form. Set $A^\times:=A\setminus A^0$, $\;\bar{A}:=A/A^0$ and let $\;\bar{}:A\longrightarrow\bar{A}$ be the canonical map.
The form $\frm$ is called positive definite (positive semidefinite) if $(\a,\a)>0$ $((\a,\a)\geq 0)$ for all $\a\in A\setminus\{0\}$. If $\frm$ is positive semidefinite, then it is easy to see that
$$A^0:=\{\alpha\in A\;|\;(\alpha,\alpha)=0\}.$$
From now on we assume that $\frm$ is a positive semidefinite form on $A$.
For a subset $B$ of $A$, let $B^\times:=B\setminus A^0$ and $B^0:=B\cap A^0$.
For $\alpha,\beta\in A$, if $(\alpha,\alpha)\not= 0$, set $(\beta,\alpha^\vee):=
2(\beta,\alpha)/(\alpha,\alpha)$ and if $(\alpha,\alpha)=0$, set $(\beta,\alpha^\vee):=0$.
A subset $X$ of $A^\times$ is called connected if it cannot be written as a disjoint union of two nonempty
orthogonal subsets. The form $\frm$ induces a unique form on $\bar{A}$ by
$$(\bar{\alpha},\bar{\beta})=(\alpha,\beta)\quad\mathrm{for}\;\alpha,\beta\in A.$$
This form is positive definite on $\bar{A}$. Thus, $\bar{A}$ is a torsion free group.
For a subset $S$ of $A$, we denote by $\langle S\rangle$, the subgroup generated by $S$.
Here is the definition of an affine reflection system from \cite{AYY}.

\begin{DEF}\cite[Definition 1.3]{AYY}\label{ARS def}
Let $A$ be an abelian group equipped with a nontrivial symmetric
positive semidefinite form $\frm$. Let $R$ be a subset of $A$. A subset $R$ of $A$ is called an affine reflection system ({\it ARS}) in $A$, if it satisfies the following 3
axioms:
\begin{itemize}
\item[(R1)] $R=-R$,
\item[(R2)] $\langle R\rangle=A$,
\item[(R3)] for $\alpha\in\rt$ and $\beta\in R$, there exist $d, u\in\mathbb{Z}_{\geq0}$ such that
$$(\beta+\mathbb{Z}\alpha)\cap R=\{\beta-d\alpha,\dots,\beta+u\alpha\}\quad\mathrm{and}\quad d-u=(\beta,\alpha^\vee).$$
\end{itemize}
The affine reflection system $R$ is called {\it irreducible} if it
satisfies
\begin{itemize}
\item[(R4)] $\rt$ is connected.
\end{itemize}
Moreover, $R$ is called {\it tame} if
\begin{itemize}
\item[(R5)] $R^0\subseteq\rt-\rt$ (elements of $R^0$ are non-isolated).
\end{itemize}
Finally $R$ is called {\it reduced} if it satisfies
\begin{itemize}
\item[(R6)] $\alpha\in\rt\Rightarrow 2\alpha\not\in\rt$.
\end{itemize}
Elements of $\rt$ (resp. $R^0$) are called {\it non-isotropic
roots} (resp. {\it isotropic roots}).
An affine reflection system $R$ in $A$ is called a {\it locally finite root system} if $A^0=\{0\}$.
\end{DEF}

Let $R$ be an affine reflection system. Since the form is nontrivial, $A\not=A^0$. The image of $R$ under $\bar{\;}$ is shown by $\bar{R}$. By \cite[Corollary 1.9]{AYY}, $\bar{R}$ in $\bar{A}$ is a locally finite root system. The {\it type} and the rank of $R$ are defined to be the type and the rank of $\bar{R}$, respectively.

Since this work is devoted to the study of Weyl groups of affine reflection systems of type $A_1$, {\it for the rest we assume that $R$ is a tame irreducible affine reflection system of type $A_1$}. By \cite[Theorem 1.13]{AYY}, $R$ contains a finite root system $\rd=\{0,\pm\ep\}$ of type $A_1$ and a subset $S\subseteq R^0$, such that
\begin{equation}\label{type-A_1}
R=(S+S)\cup(\rd^\times+S),
\end{equation}
where $S$ is a {\it pointed reflection subspace of} $A^0$, in the sense that it satisfies,
\begin{equation*}
0\in S,\quad S\pm 2S\subseteq S,\quad\mathrm{and}\quad\langle S\rangle=A^0.
\end{equation*}
According to \cite[Theorem 1.13]{AYY}, any tame irreducible affine reflection system of type $A_1$ arises in this way. In particular
\begin{equation}\label{TARS}
\rtid:=A^0\cup(\rd+A^0)
\end{equation}
is an affine reflection system in $A$. It follows easily that $\rtid$ contains any affine reflection system of type $A_1$ in $A$.

Let $\Aut(A)$ denote the group of automorphisms of $A$. For $\a\in A$, one defines $w_\a\in\Aut(A)$ by
$$w_\a(\b)=\b-(\b,\a^\vee)\a,$$ for $\beta\in A$. We call $w_\a$ the reflection based on $\a$, since it sends $\a$ to $-\a$ and
fixes pointwise the subgroup $\{\b\in A\mid (\b,\a)=0\}$ of $A$. Note that if $\a\in A^0$, then according to our convention,
$(\b,\a^\vee)=0$ for all $\b$ and so $w_\a=\hbox{id}_A$, where $\hbox{id}_A$ is the identity map on $A$.
For a subset $B$ of $R$, the subgroup of $\Aut(A)$ generated by $w_\a$, $\a\in B$, is denoted by $\w_B$. The group $\w_R$ is called the Weyl group of $R$.


Using axiom (R3) of \ref{ARS def}, we have $w(R)\subseteq R$, for $w\in\w_R$.
One can easily conclude that for $w\in\w_R$ and $\a,\b\in A$,
$$(w\a,w\b)=(\a,\b).$$
In turn this leads us to the fact that for $\a\in R$ and $w\in\w_R$,
\begin{equation}\label{conj}
ww_\a w^{-1}=w_{w(\a)}.
\end{equation}

Without loss of generality, we may assume that
\begin{equation*}
(\epsilon,\epsilon)=2.
\end{equation*}
Let $\dot A:=\langle\rd\rangle$. Then $A=\dot{A}\oplus A^0$, we also have $\langle R^0\rangle=\langle S\rangle=A^0$. Let
$p:A\rightarrow A^0$, be the projections onto $A^0$.

For each $\alpha\in A$, we have
$$\alpha=\sgn(\alpha)\epsilon+p(\alpha),$$
where $\sgn:A\longrightarrow{\mathbb Z}$ is a group epimorphism. Clearly, each $\a\in A$ is uniquely
determined by its images under the maps $p$ and $\sgn$.
Since the image of $p$ is contained in the radical of the form, the form $\frm$ is uniquely determined by the map $\sgn$, namely
\begin{equation*}
(\beta,\alpha^\vee)=(\beta,\alpha)=2\sgn(\beta)\sgn(\alpha),
\end{equation*}
for all $\a,\b\in A$.
Here we recall some results from \cite{AN}.

\begin{lem}\cite[Lemma 2.3]{AN}\label{action lem}
Let $w:=w_{\alpha_1}\cdots w_{\alpha_t}\in\w_R$, $\a_i\in R^\times$.
Then for $\beta\in R$, we have
$$\sgn(w\b)=(-1)^t\sgn(\beta)\andd p(w\b)=p\big(\beta-2(-1)^t\sgn(\b)\sum_{i=1}^t(-1)^i\sgn(\alpha_i){\alpha_i}\big).$$
\end{lem}

\begin{pro}\cite[Proposition 2.6]{AN}\label{W relations}
For $\a_1,\ldots,\a_n\in R^\times$, we have $w_{\a_1}\cdots w_{\a_n}=1$ in $\w_R$, if and
only if $n$ is even and
$$\sum_{i=1}^n(-1)^i\sgn(\a_i)p(\a_i)=0.$$
In particular, if $n$ is odd, then $w^2=1$.
\end{pro}

\begin{cor}\cite[Corollary 2.8]{AN}\label{action cor}
For $\a,\b,\g\in R$, we have
$$(w_\a w_\b w_\g)^2=1.$$
In particular, for $\alpha_1,\dots,\alpha_t\in\rt$ we have
$$w_{\alpha_1}\cdots w_{\alpha_t}=w_{\alpha_1}\cdots w_{\alpha_{i-1}}w_{\alpha_{i+2}}w_{\alpha_{i+1}}w_{\alpha_i}w_{\alpha_{i+3}}\cdots w_{\alpha_t}.$$
\end{cor}

The following definition is suggested by Proposition \ref{W relations}.

\begin{DEF}\cite[Definition 2.7]{AN}\label{semi-relation}
Let $P$ be a subset of $R$. We call a  $k$-tuple $(\a_1,\ldots,\a_k)\in P$, an \it{alternating $k$-tuple in $P$} if $k$ is even
and $\sum_{j=1}^k(-1)^jsgn(\alpha_j)p(\alpha_j)=0$. We denote by $\Alt(P)$, the set of all alternating $k$-tuples in $P$ for all $k$.

By Proposition \ref{W relations}, if $(\a_1,\ldots,\a_k)\in\Alt(P)$, then $w_{\a_1}\cdots w_{\a_k}=1$ in $\w$. 
\end{DEF}

The following proposition is a generalization of \cite[Proposition 4.4]{AN} for $\rtid=A^0\cup(\pm\ep+A^0)$.

\begin{pro}\label{baby W}
We have
\begin{itemize}
\item[(i)] $w_{\alpha+\sigma+\delta}=w_{\alpha+\sigma}w_\alpha
w_{\alpha+\delta}$, for $\alpha\in\rtid$ and $\sigma,\delta\in
\rtid^0$.
\item[(ii)]
$w_{\alpha+k\sigma}w_\alpha=(w_{\alpha+\sigma}w_\alpha)^k$, $k\in\mathbb{Z}$, $\alpha\in\rtid$ and $\sigma\in\rtid^0$.
\item[(iii)]
If $\Pi^0$ is any generating subset of $A^0$ and $\a\in\rtid^\times$, then $w_\b\in\w_{\pm\a+\Pi^0}$, for any $\b\in\rtid^\times=\pm\ep+A^0$.
\item[(iv)] $\w_R=\w_{\rtid}$, for any ARS $R$ in $A$.
\end{itemize}
\end{pro}

\begin{proof}
(i)-(ii) The proof is the same as the proof of \cite[Proposition 4.1 (i)-(ii)]{AN}.

(iii) The proof is essentially the same as the proof of \cite[Proposition 4.1 (iii)]{AN}, however for the convenience of the reader we provide the details here. Let $\b$ be an arbitrary element of $\pm\ep+ A^0$. Then $\b=k\ep+\sg$, for $k\in\{\pm1\}$ and $\sigma\in A^0$. Without loss of generality, we assume that $k=\sgn(\a)$. Then we have $\b=\a+\sg-p(\a)$. Let
$$\sg-p(\a)=\sum_{\tau\in\Pi^0}n_\tau\tau,$$
where $n_\tau\in\mathbb{Z}$ and $n_\tau=0$ for all but a finite number of $\tau\in\Pi_0$. From (i) we have
$$w_\b w_\a=\prod_{\tau\in\Pi^0}w_{\a+n_\tau\tau}w_\a.$$
Now for each $\tau\in\Pi^0$, from (ii) we have $$w_{\a+n_\tau\tau}w_\a=(w_{\a+\tau}w_\a)^{n_\tau}.$$ This way we obtain an expression of $w_\b$ with respect to the reflections based on elements of $\pm\a+\Pi^0$. This means that $w_\b\in\w_{\pm\a+\Pi^0}$.

(iv) Let $R$ be an ARS of type $A_1$ in $A$. Then $R^\times=\pm\a+S$ for a pointed reflection subspace $S$ in $A^0$ and $\a\in\rtid^\times$. Since $S$ generates $A^0$, by (iii), we have
$$\w_R=\w_{R^\times}=\w_{\rtid}.$$
\end{proof}

Suggested by the proof of Proposition \ref{baby W} (iv), we have the following definition.

\begin{DEF}\label{A1WG}
$\w:=\w_{\rtid}$ is called the \textit{$A_1$-type Weyl group} on $(A,\frm)$. If there is no confusion about $\frm$, we simply call $\w$ the $A_1$ -type Weyl group on $A$.
\end{DEF}


In this part we define two maps on $\w$. These maps are crucial for the rest of this work. For $w\in\w=\langle w_\a\;|\;\a\in\rtid^\times\rangle$, suppose that $w_{\a_1}\cdots w_{\a_n}$ is an expression of $w$ with respect to $\rtid^\times$, i.e., $\a_i\in\rtid^\times$, for $1\leq i\leq n$.
We define $\vep(w)=(-1)^n$. If $w=w_{\b_1}\cdots w_{\b_m}$ is another expression of $w$ with respect to $\rtid^\times$, then from Lemma \ref{action lem} we have
$$(-1)^msgn(\g)=sgn(w\g)=(-1)^nsgn(\g),$$
for any $\g\in\rtid^\times$. Thus the map $\vep$ from $\w$ to the multiplicative group $\{-1,1\}$ is well defined.
For $w,w'\in\w$, let $w=w_{\a_1}\cdots w_{\a_n}$ and $w'=w_{\a'_1}\cdots w_{\a'_{n'}}$, where $\a_i,\a_j'\in\rtid^\times$ for $1\leq i\leq n$ and $1\leq j\leq m$. We have
$$\vep(ww')=(-1)^{n+n'}=\vep(w)\vep(w').$$
So the map $\vep:\w\longrightarrow\{-1,1\}$ is a group homomorphism. For $w\in\w$, we call $\vep(w)$ the sign of $w$. In this section, we call $w\in\w$ an even (odd) element, if $\vep(w)=1$ ($\vep(w)=-1$).

Again for $w\in\w$, suppose that $w_{\a_1}\cdots w_{\a_n}$ is an expression of $w$ with respect to $\rtid^\times$. We define
\begin{equation}\label{zpartexp}
\ta_w:=\sum_{i=1}^n(-1)^{n-i}\sgn(\a_i)p(\a_i)=\vep(w)\sum_{i=1}^n(-1)^i\sgn(\a_i)p(\a_i)\in A^0.
\end{equation}
If $w_{\b_1}\cdots w_{\b_m}$ is another expression of $w$ with respect to $\rtid^\times$, then from Lemma \ref{action lem} we have
$$\vep(w)\sum_{j=1}^m(-1)^{j}\sgn(\b_j)p(\b_j)=\frac{1}{2sgn(\g)}p(\g-w\g)=\vep(w)\sum_{i=1}^n(-1)^{i}\sgn(\a_i)p(\a_i),$$
for $\g\in\rt$. So, $\ta_w$ is independent of the choice of the expression for $w$. Thus the map
\begin{equation*}
\begin{array}{ll}
T:\w\longrightarrow A^0\\
\quad\;\;\; w\longmapsto\ta_w\\
\end{array}
\end{equation*}
is well defined. Recall that by Proposition \ref{baby W}, $w_{\pm\ep+\sg}\in\w$, for $\sg\in A^0$. Now, since $T(w_{\ep+\sg})=\sg$, for $\sg\in A^0$, $T$ is an onto map. However, $T$ is not one-to-one. To see this, note that $T(w_{-\ep+\sg_1}w_{\ep+\sg_2})=\sg_1+\sg_2=T(w_{\ep+\sg_1+\sg_2})$, but $w_{-\ep+\sg_1}w_{\ep+\sg_2}\not=w_{\ep+\sg_1+\sg_2}$.

Let $w\in\w$ and $w_{\a_1}\dots w_{\a_n}$ be an expression of $w$ with respect to $\rtid^\times$. From Lemma \ref{action lem}, we have
$$w(\a)=(-1)^n\sgn(\a)\ep+p(\a)-2\sgn(\a)\sum_{i=1}^n(-1)^{n-i}\sgn(\a_i)p(\a_i).$$
Comparing this equation with definitions of the maps $\vep$ and $T$, we conclude that
\begin{equation}\label{waction}
w(\a)=\vep(w)\sgn(\a)\ep+p(\a)-2\sgn(\a)T(w),
\end{equation}
for $\a\in\rtid$.

\begin{lem}\label{antihom}
For $w_1,w_2\in\w$, we have
$$T(w_1w_2)=\vep(w_2)T(w_1)+T(w_2).$$
\end{lem}

\begin{proof}
Let $w_{\a_1}\dots w_{\a_m}$ and $w_{\b_1}\dots w_{\b_n}$ be expressions of $w_1$ and $w_2$ with respect to $\rtid^\times$, respectively. Then
\begin{eqnarray*}
T(w_1w_2)&=&\sum_{i=1}^m(-1)^{m+n-i}\sgn(\a_i)p(\a_i)+\sum_{j=1}^n(-1)^{m+n-(m+j)}\sgn(\b_j)p(\b_j)\\
&=&\vep(w_2)\sum_{i=1}^m(-1)^{m-i}\sgn(\a_i)p(\a_i)+\sum_{j=1}^n(-1)^{n-j}\sgn(\b_j)p(\b_j)\\
&=&\vep(w_2)T(w_1)+T(w_2).
\end{eqnarray*}
\end{proof}

Lemma \ref{antihom} shows that, in general, $T$ is not a group homomorphism.

\begin{pro}\label{WRtCorr}
The map $w\longmapsto\vep(w)(\ep+T(w))$ is a one-to-one correspondence from $\w$ onto $\pm\ep+A^0$.
\end{pro}

\begin{proof}
We define the map $\mathfrak{s}:\w\longrightarrow\pm\ep+A^0$ by $\mathfrak{s}(w)=\vep(w)(\epsilon+T(w))$. For $\a\in\rtid^\times$, let $w=w_\a$ if $\sgn(\a)=-1$ and let $w=w_\ep w_\a$ if $\sgn(\a)=1$. Then $w\in\w$ and $\mathfrak{s}(w)=\a$. So $\mathfrak{s}$ is an onto map. On the other hand, let $\mathfrak{s}(w_1)=\mathfrak{s}(w_2)$, for $w_1,w_2\in\w$. Then $\vep(w_1)\ep=\vep(w_2)\ep$ and $\vep(w_1)T(w_1)=\vep(w_2)T(w_2)$. So we have $\vep(w_1)=\vep(w_2)$ and $T(w_1)=T(w_2)$. Then by (\ref{waction}), For $\a\in\pm\ep+A^0$, we have
\begin{eqnarray*}
w_1\a &=& \vep(w_1)\sgn(\a)\ep+p(\a)-2\sgn(\a)T(w_1)\\
&=& \vep(w_2)\sgn(\a)\ep+p(\a)-2\sgn(\a)T(w_2)\\
&=& w_2\a.
\end{eqnarray*}
It means that $w_1=w_2$. Thus $\mathfrak{s}$ is one-to-one.
\end{proof}

Let $\w^0:=\Ker(\vep)$. Then $\w^0$ is the set of all even elements of $\w$.

\begin{pro}\label{W0A0iso}
$T_{|_{\w^0}}:\w^0\longrightarrow A^0$ is an isomorphism.
\end{pro}

\begin{proof}
Let $w_1,w_2$ be elements of $\w^0$. Then by Lemma \ref{antihom}
$$T(w_1w_2)=\vep(w_2)T(w_1)+T(w_2)=T(w_1)+T(w_2).$$
Thus $T_{|_{\w^0}}$ is a homomorphism. If $w\in\Ker(T_{|_{\w^0}})$, then $\vep(w)=1$ and $T(w)=0$. Thus by (\ref{waction}),  $w=1$. For $\sg\in A^0$, let $w=w_\ep w_{\ep+\sg}\in\w^0$. Then $T(w)=\sg$. So, $T_{|_{\w^0}}$ is an isomorphism.
\end{proof}

\begin{thm}\label{uniexpspro}
Each element $w$ in $\w$ has a unique expression in the form
$$w=w_\ep^{\d_{\vep(w),1}}w_{\ep+T(w)},$$
where $\d$ is the Kronecker delta. In particular, if $\wdt:=\langle w_\ep\rangle$ is the finite Weyl group of type $A_1$, we have $$\w=\wdt\ltimes\w^0.$$
\end{thm}

\begin{proof}
To prove the first equation it is enough to show, by Proposition \ref{WRtCorr}, that
$$\vep(w)=\vep(w_\ep^{\d_{\vep(w),1}}w_{\ep+T(w)})\andd T(w)=T(w_\ep^{\d_{\vep(w),1}}w_{\ep+T(w)}).$$
Let $t=\d_{\vep(w),1}$. By definition of $\vep$, we have $\vep(w_\ep^tw_{\ep+T(w)})=-1$, if $\vep(w)=-1$ and $\vep(w_\ep^tw_{\ep+T(w)})=1$, if $\vep(w)=1$. Thus $\vep(w)=\vep(w_\ep^tw_{\ep+T(w)})$.

By Lemma \ref{antihom}, we have
$$T(w_\ep^tw_{\ep+T(w)})=-\d_{\vep(w),1}0+T(w)=T(w).$$
Now, for each $w\in\w$, we have $w=w_\ep^{t+1}w_\ep w_{\ep+T(w)}$. Thus $\w=\wdt\w^0$. This together with the facts that $\wdt\bigcap\w^0=\{0\}$ and $\w^0=\Ker(\vep)$ is a normal subgroup of $\w$ show that $\w=\wdt\ltimes\w^0$.
\end{proof}

\begin{rem}
(i) By Theorem \ref{uniexpspro}, every element of $\w^0$ is of the form $w_\ep w_{\ep+\sigma}$, for some $\sigma\in A^0$ and by Proposition \ref{W0A0iso}, we have $T(w_\ep w_{\ep+\sigma})=\sigma$. Let $\varphi:\wdt\longrightarrow\Aut(A^0)$ be defined by $\varphi(w^t_\ep)(\sigma)=T_{|_{\w^0}}(w_\ep^t w_\ep w_{\ep+\sigma}w_\ep^t)$. Then it is easy to see that $\w\cong\wdt\ltimes_{\varphi} A^0$, i.e., $\w$ is isomorphic to the semidirect product of the finite Weyl group of type $A_1$ and the radial of the form $\frm$ on $A$.

(ii) Let $\w^1=\{w_\alpha\;|\;\alpha\in\ep+A^0\}$. From Theorem \ref{uniexpspro}, we have $\w=\w^0\uplus\w^1$, where $\uplus$ means disjoint union. Both $\w^0$ and $\w^1$ are in one-to-one correspondence with $A^0$. Thus one can consider $\w$ as a union of two copies of $A^0$.
\end{rem}

At the end of this section, using (\ref{waction}) and the unique expression of Theorem \ref{uniexpspro}, we prove some interesting identities.

\begin{lem}\label{identities}
Let $w,w_1,w_2\in\w$. We have
\begin{itemize}
\item[(i)] $w^{-1}=w_\ep^{\d_{\vep(w),1}}w_{\ep-\vep(w)T(w)}=w_\ep^{\d_{\vep(w),1}}w_{\vep(w)\ep-T(w)}$.
\item[(ii)]
      $w_1w_2w_1^{-1}=w_\ep^{\d_{\vep(w_2),1}}w_{\ep+\vep(w_1)(T(w_2)+(\vep(w_2)-1)T(w_1))}.$
\end{itemize}
\end{lem}

\begin{proof}
(i) Since $\vep$ is a group homomorphism, we have $\vep(w^{-1})=\vep(w)$. Also, from Lemma \ref{antihom}, we have
$$0=T(ww^{-1})=\vep(w)T(w)+T(w^{-1}).$$
Thus $T(w^{-1})=-\vep(w)T(w)$. The second equation is obvious.

(ii) Again, it is obvious that $\vep(w_1w_2w_1^{-1})=\vep(w_2)$. From Lemma \ref{antihom} and (i), we have
\begin{eqnarray*}
T(w_1w_2w_1^{-1})&=&\vep(w_1^{-1})T(w_1w_2)+T(w_1^{-1})\\
&=&\vep(w_1)(\vep(w_2)T(w_1)+T(w_2))-\vep(w_1)T(w_1)
\end{eqnarray*}
Thus $T(w_1w_2w_1^{-1})=\vep(w_1)(T(w_2)+(\vep(w_2)-1)T(w_1))$. From Theorem \ref{uniexpspro}, we get the result.
\end{proof}

\section{\bf A length function for the $A_1$-type Weyl group $\w$ of nullity $\nu$}\setcounter{thm}{0}\label{wlf}
In this section we assume that $A$ is a free abelian group of rank $\nu+1$ and we offer a length function for the $A_1$-type Weyl group $\w$ on $A$. We call $\w$ the $A_1$-type Weyl group of nullity $\nu$. Set $\Lam:=A^0$.

Let $R$ be an affine reflection system of type $A_1$ in $A$. We identify $R$ with $1\otimes R$ in $$\v:=\mathbb{R}\otimes_\mathbb{Q}A.$$
Then $R$ turns out to be an extended affine root system of type $A_1$ in $\v$ in the sense of \cite[Definition II.2.1]{AABGP}. We simply say that $R$ is an extended affine root system in $A$. From (\ref{type-A_1}), $R=(S+S)\cup(\rd^\times+S)$, where, in this case, the pointed reflection space $S$ is a \textit{semilattice} in $\Lam$ in the sense of
\cite[Definition II.1.2]{AABGP}, namely $S$ is a subset of $\Lam$ satisfying
$$0\in S,\quad S\pm 2S\subseteq S,\quad \langle S\rangle=\Lam.$$
By \cite[Remark II.1.6]{AABGP}, we have
$$S=\bigcup_{i=0}^m(\tau_i+2\Lambda),$$
where $\tau_0=0$ and $\tau_i$'s represent distinct cosets of $2\Lam$ in $\Lam$, for $0\leq i\leq m$. The positive integer $m$ is called the index of $S$. By \cite[Proposition II.1.11]{AABGP}, $\Lambda$ has a $\mathbb{Z}$-basis consisting of elements of $S$.

We assume that $\ba=\dot{\ba}\cup\ba^0$ is a fixed $\mathbb{Z}$-basis of $A$, where $\ba^0=\{\sg_1,\dots,\sg_\nu\}$ is a $\mathbb{Z}$-basis of $\Lam$ and $\dot{\ba}=\{\ep\}$ is a $\mathbb{Z}$-basis of the fixed complement $\ad$ of $\Lam$. The extended affine root system $$\rtid=\Lam\cup(\pm\ep+\Lam)$$ in $A$ is called the \textit{toroidal} root system of type $A_1$. Let $$S_b:=\bigcup_{i=0}^\nu(\sg_i+2\Lam),$$
where $\sg_0=0$. It is clear that $S_b$ is a semilattice in $A^0$. We call the extended affine root system
$$R_b:=(S_b+S_b)\cup(\pm\ep+S_b)$$
the \textit{baby} extended affine root system of type $A_1$. Up to isomorphism, we may assume that for any extended affine root system $R$ of type $A_1$ in $A$, we have $\ep\in R^\times$ and $\ba^0\subseteq S$, where $S$ is the semilattice associated with $R$. Then $R_b\subseteq R\subseteq\rtid$. We have
$$\rtid=\{k\ep+\sg\;|\;k\in\{0,\pm1\},\;\sg\in\Lam\}.$$

For $\sg=\sum_{i=1}^\nu m_i\sg_i\in\rtid^0=\Lam$, let
\begin{equation}\label{pmp}
m^+:=\sum_{\substack{i=1\\m_i>0}}^\nu m_i\andd m^-:=\sum_{\substack{i=1\\m_i<0}}^\nu m_i,
\end{equation}
where the sum on an empty set is considered to be zero.
Then, we define the \textit{height function} $\hti_{\ba^0}:\Lam\longrightarrow\mathbb{Z}$ as follows. For $\sg\in\Lam$, let
\begin{equation}\label{zht}
\hti_{\ba^0}(\sg):=\left\{
        \begin{array}{ll}
            2m^+, & \hbox{$m^+\geq|m^-|$,}\\
            2m^-, & \hbox{$m^+<|m^-|$.}
        \end{array}\right.
\end{equation}
Now, we extend the height function to $\rtid$ as follows.

\begin{DEF}\label{ht}
For $\a=k\ep+\sigma\in\rtid$, we define
\begin{equation*}
\hti(\a):=\left\{
    \begin{array}{ll}
        k-\hti_{\ba^0}(\sg), & \hbox{$k=-1,\;m^+=|m^-|$,}\\
        k+\hti_{\ba^0}(\sg), & \hbox{otherwise.}
    \end{array}\right.
\end{equation*}
If $\a\in\rtid$ is a nonzero root, then $\hti(\a)\not=0$.
\end{DEF}

Let $\frm_E:\Lam\times\Lam\longrightarrow\mathbb{Z}$ be a form on $\Lam$ defined by $$(\sum_{i=1}^\nu m_i\sg_i,\sum_{i=1}^\nu n_i\sg_i)_E:=\sum_{i=1}^\nu m_in_i,$$
where $m_i,n_i\in\mathbb{Z}$, for $1\leq i\leq\nu$. If one considers the $\nu$ dimensional vector space $\v^0:=\mathbb{R}\otimes_{\mathbb{Q}}\Lam$ as an Euclidean space, then $\frm_E$ is the restriction of the Euclidean inner product on $\v^0$ to $\Lam$. Using $\frm_E$, we define a notion of positive and negative roots on $\rtid$.

\begin{DEF}\label{PNroot}
Let $\a\in\rtid\setminus\{0\}$.

(i) We call $\a$ positive, if
$$sgn(\a)\not=-1\quad\mathrm{and}\quad (p(\a),\sum_{i=1}^\nu\sg_i)_E\geq0$$
or
$$sgn(\a)=-1\quad\mathrm{and}\quad (p(\a),\sum_{i=1}^\nu\sg_i)_E>0.$$
We denote the set of positive roots by $\rtid^+$.

(ii) We call a non-zero root $\a$ negative, if $\a\in\rtid\setminus\rtid^+$. We denote the set of negative roots by $\rtid^-$.
\end{DEF}

We have $\rtid=\rtid^+\cup\{0\}\cup\rtid^-$, 
$\rtid^\times=\rtid^{\times^+}\cup\rtid^{\times^-}$ and $\rtid^0=\rtid^{0^+}\cup\{0\}\cup\rtid^{0^-}$, where $\rtid^{\times^+}:=\rtid^\times\cap\rtid^+$, $\rtid^{\times^-}:=\rtid^\times\cap\rtid^-$, $\rtid^{0^+}:=\rtid^0\cap\rtid^+$ and $\rtid^{0^+}:=\rtid^0\cap\rtid^+$.

\begin{lem}\label{hipn}
Let $\a\in\rtid$. We have $\a\in\rtid^+$ (resp. $\a\in\rtid^-$) if and only if $\hti(\a)>0$ (reps. $\hti(\a)<0$).
\end{lem}

\begin{proof}
Let $\a=\sgn(\a)\ep+p(\a)\in\rtid$ and $p(\a)=\sum_{i=1}^\nu m_i\sg_i$. We have
$$(p(\a),\sum_{i=1}^\nu\sg_i)_E=\sum_{i=1}^\nu m_i=m^++m^-,$$
where $m^+$ and $m^-$ are as in (\ref{pmp}). Since $\rtid^-\subseteq-\rtid^+$ and $\hti(\a)=-\hti(-\a)$, for $\a\in\rtid^-$, it is enough to show that $\a\in\rtid^+$ if and only if $\hti(\a)>0$.

Let $\a\in\rtid^+$. From Definition \ref{PNroot} (i), if $\sgn(\a)\not=-1$, $m^++m^-\geq0$ or if $\sgn(\a)=-1$, $m^++m^->0$. Thus $m^+\geq -m^-$ or $m^+>-m^-$. From Definition \ref{ht}, we have $\hti(\a)=\sgn(\a)+2m^+>0$.

Let $\a$ be an element of $\rtid$ for which $\hti(\a)>0$. Thus, from Definition \ref{ht}, we have $\hti(\a)=\sgn(\a)+\hti(p(\a))$ and $\hti(p(\a))\geq0$. Thus $m^+\geq -m^-$, if $\sgn(\a)\not=-1$ or $m^+>-m^-$, if $\sgn(\a)=-1$.
\end{proof}

For $\sg\in\rtid^0\setminus\{0\}$, let $\sg=\sum_{i=1}^\nu m_i\sg_i$. We call $\sg$ a \textit{strictly positive isotropic root} (resp. \textit{strictly negative isotropic root}), if $m_i\geq0$ (resp. $m_i\leq0$), for $1\leq i\leq\nu$.

\begin{DEF}\label{RootBasis}
Let $R$ be an extended affine root system in $A$ and $\Pi=\{\a_0,\a_1,\dots,\a_\nu\}$ be a $\mathbb{Z}$-basis of $A$ in $R_b^\times$. We call $\Pi$ a \textit{root basis for $R$} if each strictly positive isotropic root $\sg$ in $R$ can be written in the form $\sg=\sum_{i=0}^\nu m_i\a_i$, where $m_i\geq0$.
\end{DEF}


Let $\Pi_0$ be the set of all elements of $\rtid$ with height 1. We have $$\Pi_0=\{\a_0:=\ep,\a_1:=\sg_1-\ep,\dots,\a_\nu:=\sg_\nu-\ep\},$$
where $\ba=\{\ep,\sg_1,\dots,\sg_\nu\}$ is the fixed $\mathbb{Z}$-basis of $A$. It is easy to show that $\Pi_0$ is a root basis for $\rtid$. Since the elements of $\Pi_0$ are the only elements of $\rtid$ with height 1, we call it \textit{the fundamental root basis for $R$}. We calculate the length function for $\w$ with respect to the fundamental root basis $\Pi_0$ using the height function of Definition \ref{ht}.

\begin{pro}\label{coht}
Let $\a=\sum_{i=0}^\nu n_i\a_i\in\rtid$. We have
$$|\hti(\a)|=\sum_{i=0}^\nu|n_i|.$$
In particular $\sum_{i=0}^\nu|n_i|$ is odd, if $\a\in\rtid^\times$ and $\sum_{i=0}^\nu|n_i|$ is even, if $\a\in\rtid^0$.
\end{pro}

\begin{proof}
Since $\a\in\rtid$, we have $\a=k\ep+\sg$, where $k\in\{0,\pm1\}$ and $\sg\in\Lam$. Let $\sg=\sum_{i=1}^\nu m_i\sg_i$. On the other hand, we have $\sg_i=\a_0+\a_i$ and $\ep=\a_0$. Thus
$$\a=(k+\sum_{i=1}^\nu m_i)\a_0+\sum_{i=1}^\nu m_i\a_i.$$
Since $\Pi_0$ is a $\mathbb{Z}$-basis for $A$, we have $n_0=k+\sum_{i=1}^\nu m_i=k+m^++m^-$ and $n_i=m_i$, for $1\leq i\leq\nu$. We consider three cases.

First we assume that $k=-1$ and  $m^+=|m^-|$. Then $n_0=-1$. Thus $$\sum_{i=0}^\nu|n_i|=1+m^++|m^-|=2m^++1=-(-2m^+-1)=|\hti(\a)|.$$

Next we assume that either $m^+>|m^-|$, or $m^+=|m^-|$ and $k\not=-1$. Then $m^++m^-=m^+-|m^-|\geq0$. Thus $n_0=k+m^++m^-\geq0$. So, we have
$$\sum_{i=0}^\nu|n_i|=k+m^++m^-+m^+-m^-=2m^++k=|\hti(\a)|.$$

Finally, we assume that $m^+<|m^-|$. Then $n_0=k+m^++m^-\leq0$. Thus
$$\sum_{i=0}^\nu|n_i|=-k-m^+-m^-+m^+-m^-=-2m^--k=|\hti(\a)|.$$
\end{proof}

Here we need to recall the definition of the length function of an arbitrary group, with respect to some generating subset. Let $G$ be a group and $\{g_\a\}_{\a\in\Pi}$ be a set of generators for $G$, namely
$$G=\langle g_\alpha\;|\;\alpha\in\Pi\rangle.$$
Then we have the following definition.

\begin{DEF}\label{Length}
(i) An expression of an element $g$ in $G\setminus\{1\}$, with respect to $\Pi$, is a sequence $g_{\alpha_1}\cdots g_{\alpha_k}$ which equals to $g$, for $\alpha_i\in\Pi$. The positive integer $k$ is called the length of the expression. In this case we say that this expression has occupied $k$ {\it positions}, in which $g_{\a_j}$ is in the $j$-th position. A position is called {\it odd}  ({\it even}) if $j$ is odd (even).

(ii) The length of $g\in G$ with respect to $\Pi$, which is denoted by $\ell_{\Pi}(g)$, is the smallest length of any expression of $g$. By convention, $\ell_\Pi(1)=0$.

(iii) The function $\ell_\Pi$ from $G$ to the set of non-negative integers  $\mathbb{Z}_{\geq0}$, which assigns to each $g\in G$, the integer $l_\Pi(g)$ is called the length function of $G$ with respect to $\Pi$. If there is no confusion about $\Pi$, we show the length function by $\ell$.

(iv) Any expression of $g\in G$ with length $\ell_\Pi(g)$ is called a reduced expression of $g$ with respect to $\Pi$.
\end{DEF}

Here are some elementary properties of the length function, which only depend on the above definition (See \cite[5.2]{Hu}). For $g,g'\in G$, we have

\begin{itemize}
\item[-] $\ell(g)=\ell(g^{-1})$.
\item[-] $\ell(g)=1$ if and only if $g=g_\alpha$, for $\alpha\in\Pi$.
\item[-] $\ell(g)-\ell(g')\leq\ell(gg')\leq \ell(g)+\ell(g')$.
\end{itemize}

Considering an inner automorphism of $G$, one can see the following elementary result.

\begin{lem}\label{conlen}
Let $g_0$ be an element of $G$, $\Pi'$ be a set and $\phi$ be a one to one map from $\Pi$ onto $\Pi'$. For each $\a\in\Pi$, let $g_{\phi(\a)}:=g_0g_\a g_0^{-1}$. Then $\{g_\b\}_{\b\in\Pi'}$ is a set of generators for $G$ and, for each $g\in G$, we have
$$\ell_{\Pi'}(g)=\ell_{\Pi}(g_0^{-1}gg_0).$$
\end{lem}

Now, using the concept of the height and the maps $\vep$ and $T$, we offer a formula for the length function for $\w$ with respect to $\Pi_0$. Let $w\in\w$. Until the end of this section, any expression of $w$ is considered with respect to $\Pi_0$. Let $w_{\a_{i_1}}\dots w_{\a_{i_k}}$ be an expression of $w$. For $0\leq i\leq\nu$, we denote the number of $w_{\a_i}$'s for which $w_{\a_i}$ is in an odd (resp. even) position
by $P_i$ (resp. $N_i$). We have
\begin{equation}\label{PNlen}k=\sum_{i=0}^\nu(P_i+N_i).\end{equation}

\begin{lem}\label{refNum}
Let $w\in\w$ and $w_{\a_{j_1}}\cdots w_{\a_{j_k}}$ be an expression of $w$. If $T(w)=\sum_{i=1}^\nu m_i\sg_i$ then
$$m_i=(-1)^k(P_i-N_i)=\vep(w)(P_i-N_i).$$
\end{lem}

\begin{proof}
From (\ref{zpartexp}), we have
\begin{eqnarray*}
T(w)&=&T(w_{\a_{j_1}}\cdots w_{\a_{j_k}})\\
&=&\sum_{s=1}^k(-1)^{k-s}\sgn(\a_{j_s})p(\a_{j_s}).
\end{eqnarray*}
For $1\leq i\leq\nu$, let $p_i:\v^0\longrightarrow\mathbb{R}$ be $i$th coordinate function with respect to $\ba^0$. Thus $p_i(\sg_j)=\d_{ij}$, where $\d$ is  the Kronecker delta function. Then we have
\begin{eqnarray*}
m_i&=&p_i(T(w))\\
&=&\sum_{s=1}^k(-1)^{k-s}\sgn(\a_{j_s})\d_{ij_s}\\
&=&(-1)^k\sum_{s=1}^k(-1)^{s-1}\d_{ij_s}\\
&=&(-1)^k(P_i-N_i).
\end{eqnarray*}
\end{proof}

\begin{lem}\label{reflenbund}
Let $\sg=\sum_{i=1}^\nu m_i\sg_i\in\Lam$ and $w_{\a_{j_1}}\cdots w_{\a_{j_k}}$ be an expression of $w_{\ep+\sg}$.
\begin{itemize}
\item[(i)] If $m^+\geq-m^-$, then $|\hti(\sg)|\leq k-1$.
\item[(ii)] If $m^+<-m^-$, then $|\hti(\sg)|\leq k+1$.
\end{itemize}
\end{lem}

\begin{proof}
Since $\vep(w_{\ep+\sg})=-1$, $k$ is odd. Since for every odd position in the expression, except for the last one, there is an even position, we have $\sum_{i=0}^\nu P_i=\sum_{i=0}^\nu N_i+1$.

(i) Let $J$ be the set of $0\leq i\leq \nu$ for which $m_i>0$. From Definition \ref{ht} and Lemma \ref{refNum}, we have
$$|\hti(\sg)|=2m^+=2(-1)^k\sum_{i\in J}(P_i-N_i)=2\sum_{i\in J}(N_i-P_i)\leq 2\sum_{i\in J}N_i\leq2\sum_{i=0}^\nu N_i.$$
Since $\sum_{i=0}^\nu P_i=\sum_{i=0}^\nu N_i+1$, using (\ref{PNlen}), we have $2\sum_{i=0}^\nu N_i\leq k-1$. Thus
$$|\hti(\sg)|\leq2\sum_{i=0}^\nu N_i\leq k-1.$$

(ii) Let $J$ be the set of $0\leq i\leq \nu$ for which $m_i<0$. From Definition \ref{ht} and Lemma \ref{refNum}, we have
$$|\hti(\sg)|=-2m^-=-2(-1)^k\sum_{i\in J}(P_i-N_i)=2\sum_{i\in J}(P_i-N_i)\leq 2\sum_{i\in J}P_i\leq2\sum_{i=0}^\nu P_i.$$
Since $\sum_{i=0}^\nu P_i=\sum_{i=0}^\nu N_i+1$, using (\ref{PNlen}), we have
$$|\hti(\sg)|\leq2\sum_{i=0}^\nu P_i=\sum_{i=0}^\nu P_i+\sum_{i=0}^\nu N_i+1\leq k+1.$$
\end{proof}

\begin{lem}\label{wlenbund}
Let $\sg=\sum_{i=1}^\nu m_i\sg_i\in\Lam$ and $w_{\a_{j_1}}\cdots w_{\a_{j_k}}$ be an expression of $w=w_\ep w_{\ep+\sg}$. Then we have $$|\hti(\sg)|\leq k.$$
\end{lem}

\begin{proof}
Since $\vep(w)=1$, $k$ is even. Since for each odd (even) position in the expression of $w$ there is an even (odd) position, we have $\sum_{i=0}^\nu P_i=\sum_{i=0}^\nu N_i$. Notice that $T(w)=\sg$. First let $m^+\geq-m^-$ and $J$ be the set of $0\leq i\leq \nu$ for which $m_i>0$. By Lemma \ref{refNum}, using (\ref{PNlen}), we have
$$|\hti(\sg)|=2m^+=2\sum_{i\in J}(P_i-N_i)\leq2\sum_{i\in J}P_i\leq2\sum_{i=0}^\nu P_i=k.$$
Next let $m^+<-m^-$ and $J$ be the set of $0\leq i\leq \nu$ for which $m_i<0$. Again, by Lemma \ref{refNum}, using (\ref{PNlen}), we have
$$|\hti(\sg)|=-2m^-=-2\sum_{i\in J}(P_i-N_i)=2\sum_{i\in J}(N_i-P_i)\leq2\sum_{i\in J}N_i\leq2\sum_{i=0}^\nu N_i=k.$$
\end{proof}

\begin{pro}\label{reflectionL}
For $\a\in\rtid^\times$, we have
$$\ell_{\Pi_0}(w_\a)=|\hti(\a)|.$$
In particular, if $\a=\sum_{i=0}^\nu n_i\a_i$, we have
$\ell_{\Pi_0}(w_\a)=\sum_{i=0}^\nu|n_i|$.
\end{pro}

\begin{proof}
We have $w_\a=w_{-\a}$ and $|\hti(\a)|=|\hti(-\a)|$. Thus without loss of generality we may assume that $\sgn(\a)=1$ or equivalently $\a=\ep+\sg$, for some $\sg\in\Lam$. First we show that $w_\a$ has an expression of length $|\hti(\a)|$ with respect to $\Pi_0$. One should notice that $\vep(w_\a)=-1$, where $\vep$ is the homomorphism defined in Section \ref{ARS}. This means that length of any expression of $w_\a$ is odd. Also, from Definition \ref{ht}, we have $|\hti(\a)|$ is odd. Let
$$\sg=\sum_{i=1}^\nu m_i\sg_i,$$
as in (\ref{zht}) and let $J^+$ (resp. $J^-$) be the set of $1\leq i\leq\nu$ for which $m_i>0$ (resp. $m_i<0$). Then we have
$$m^+=\sum_{i\in J^+}m_i\quad\mathrm{and}\quad m^-=\sum_{i\in J^-}m_i.$$ To build an expression for $w_\a$ of required length, we consider two cases $m^+\geq|m^-|$ and $m^+<|m^-|$ separately.

Case(i) $m^+\geq|m^-|$. Let $k:=|\hti(\a)|=2m^++1$. We introduce an expression $w_{\b_1}\cdots w_{\b_k}$, where $\b_i\in\Pi_0$, $1\leq i\leq k$, as follows. Since $k=2m^++1$, we must have $m^++1$ odd and $m^+$ even positions in the expression. For each $i\in J^+$, we choose $m_i$ even positions and in each such position $r$, we put $\b_r=\a_i$. Since the number of even positions is $m^+=\sum_{i\in J^+}m_i$, there is no even unfilled position in the expression. For each $i\in J^-$, we choose $-m_i$ odd positions and for each such position $r$, we put $\b_r=\a_i$. Since the number of odd positions which are filled, is $|m^-|=\sum_{i\in J^-}-m_i$ and $|m^-|\leq m^+<m^++1$, there is at least one unfilled odd position. For any such position $r$, we put $\b_r=\a_0$.

Case(ii) $m^+<|m^-|$. Let $k:=|\hti(\a)|=-2m^--1$. Again, we introduce an expression $w_{\b_1}\cdots w_{\b_k}$, where $\b_i\in\Pi_0$, $1\leq i\leq k$, as follows. Since $k=-2m^--1$, we must have $-m^-$ odd and $-m^--1$ even positions in the expression. For each $i\in J^-$, we choose $-m_i$ odd positions and in each such position $r$, we put $\b_r=\a_i$. Since the number of odd positions is $-m^-=\sum_{i\in J^-}-m_i$, there is no odd unfilled position in the expression. For each $i\in J^+$, we choose $m_i$ even positions and for each such position $r$, we put $\b_r=\a_i$. Up until now, the number of filled even positions is $m^+$. On the other hand, $m^+<|m^-|$ and the total number of even positions in the expression is $|m^-|-1$. Thus, if $|m^-|\not=m^++1$, there is at least one even unfilled position in the expression. For any such position $r$, we put $\b_r=\a_0$.

From (\ref{waction}), in both cases we have
\begin{equation*}
w(\ep)=-\ep-2(\sum_{i\in J^+}m_i\sg_i+\sum_{i\in J^-}m_i\sg_i)=-\ep-2\sg=w_\a(\ep).
\end{equation*}
Thus we have $w=w_\a$, i.e., the given expressions are expressions of $w_\a$ in both cases.

Now we show that $w_\a$ does not have an expression of smaller length. Assume that $w_\a$ has an expression of length $p$, with respect to $\Pi_0$, where $p<|\hti(\a)|$. So $p\leq|\hti(\a)|-2$. If $m^+\geq-m^-$, then $|\hti(\a)|=|\hti(\sg)|+1$. Thus $p\leq|\hti(\a)|-2=|\hti(\sg)|-1$. On the other hand, by Lemma \ref{reflenbund} (i), we have $p\geq|\hti(\sg)|+1$. This is a contradiction. If $m^+<-m^-$, then $|\hti(\a)|=|\hti(\sg)|-1$.  Thus $p\leq|\hti(\a)|-2=|\hti(\sg)|-3$. Again, by Lemma \ref{reflenbund} (ii), we have $p\geq|\hti(\sg)|-1$, which is a contradiction.
\end{proof}

Recall that according to Theorem \ref{uniexpspro}, each element $w$ of $\w$ has an expression of the form $w_\ep^{\d_{\vep(w),1}}w_{\ep+T(w)}$, where $\vep$ and $T$ are the maps defined in Section \ref{ARS}.

\begin{thm}\label{LFT}
For $w\in\w$, we have
$$\ell_{\Pi_0}(w)=|\hti(\ep+T(w))|-\sgn(\hti(\ep+T(w)))t,$$
where $t=\d_{\vep(w),1}$, and $\sgn(\hti(\ep+T(w)))$ is the sign of $\hti(\ep+T(w))$ as an integer.
\end{thm}

\begin{proof}
By Theorem \ref{uniexpspro}, for $t=0$ we have $w=w_{\ep+T(w)}$. This is the case which is discussed in Proposition \ref{reflectionL}. So, let $t=1$. In this case $\vep(w)=1$ and therefore $\ell_{\Pi_0}(w)$ is even. Let $w_{\a_{i_1}}\dots w_{\a_{i_k}}$ be the reduced expression for $w_{\ep+T(w)}$ given in the proof of Proposition \ref{reflectionL}. Also, let
$$T(w)=\sum_{i=1}^\nu m_i\sg_i,\quad m^+:=\sum_{\substack{i=1\\m_i>0}}^\nu m_i\andd m^-:=\sum_{\substack{i=1\\m_i<0}}^\nu m_i,$$
as in (\ref{zht}). Then $w=w_\ep w_{\a_{i_1}}\dots w_{\a_{i_k}}$.

If $m^+\geq|m^-|$ then there is odd $1\leq j\leq k$ for which $w_{\a_{i_j}}=w_\ep$. Thus in the given expression of $w$ there is at least one $w_\ep$ in an even position. According to Corollary \ref{action cor}, we have
$$w=w_\ep w_\ep w_{\a_{i_2}}\cdots w_{\a_{i_{j-1}}}w_{\a_{i_1}}w_{\a_{i_{j+1}}}\cdots w_{\a_{i_k}}=w_{\a_{i_2}}\cdots w_{\a_{i_{j-1}}}w_{\a_{i_1}}w_{\a_{i_{j+1}}}\cdots w_{\a_{i_k}}.$$
Since $w_{\a_{i_1}}\dots w_{\a_{i_k}}$ is a reduced expression of $w_{\ep+T(w)}$, from Corollary \ref{action cor}, we have $\a_{i_{j-1}}\not=\a_{i_1}\not=\a_{i_{j+1}}$. Let $w$ has an expression of length $p$, where $p<|\hti(\ep+T(w))|-1$. Thus $p<|\hti(T(w))|$. This contradicts Lemma  \ref{wlenbund}. Thus $\ell_{\Pi_0}(w)=|\hti(\ep+T(w))|-1$.

If $m^+<|m^-|$ then there is not any $w_\ep$ in an odd position of the reduced expression of $w_{\ep+T(w)}$, which is given in Proposition \ref{reflectionL}. In this case $w$ has an expression of length $|\hti(\ep+T(w))|+1$. Let $w$ has an expression of length $p<|\hti(\ep+T(w))|+1$. Thus $p<|\hti(T(w))|-1+1=|\hti(T(w)|$, which is a contradiction compared with Lemma \ref{wlenbund}. Thus $\ell_{\Pi_0}(w)=|\hti(\ep+T(w))|+1$.
\end{proof}

\begin{cor}\label{W0len}
Let $w\in\w^0$. We have
$$\ell_{\Pi_0}(w)=|\hti(T(w))|.$$
In particular, if $T(w)=\sum_{i=0}^\nu n_i\a_i$ then $\ell_{\Pi_0}(w)=\sum_{i=0}^\nu|n_i|$.
\end{cor}

\begin{proof}
Since $\w^0=\Ker(\vep)$, thus $\vep(w)=1$. By Theorem \ref{uniexpspro}, we have $w=w_\ep w_{\ep+T(w)}$. Thus from Theorem \ref{LFT} $\ell_{\Pi_0}(w)=|\hti(\ep+T(w))|-\sgn(\hti(\ep+T(w))$. On the other hand, let $$T(w)=\sum_{i=1}^\nu m_i\sg_i,\quad m^+:=\sum_{\substack{i=1\\m_i>0}}^\nu m_i\andd m^-:=\sum_{\substack{i=1\\m_i<0}}^\nu m_i,$$
as in (\ref{zht}). We consider the following two cases.

First we assume that $m^+\geq|m^-|$. By Definition \ref{ht}, $\hti(\ep+T(w))=2m^++1$. Thus
$$|\hti(\ep+T(w))|-\sgn(\hti(\ep+T(w)))=2m^++1-1=2m^+=|\hti(T(w))|.$$

Now we assume that $|m^-|>m^+$. In this case $\hti(\ep+T(w))=2m^-+1<0$. Thus
$$|\hti(\ep+T(w))|-\sgn(\hti(\ep+T(w)))=-2m^--1+1=-2m^-=|\hti(T(w))|.$$

The last assertion in the statement follows from Proposition \ref{coht}.
\end{proof}

\begin{rem}
The proofs of Proposition \ref{reflectionL} and Theorem \ref{LFT} give rise to an algorithm for constructing expressions of smallest length for given elements of $\w$.
\end{rem}

\section{\bf Orbits and the length function}\setcounter{thm}{0}\label{rbol}
Recall that $A$ is a free abelian group of rank $\nu+1$ and $\ba=\{\ep,\sg_1,\dots,\sg_\nu\}$ is a fixed $\mathbb{Z}$-basis for $A$.
Let $\Pi=\{\a_0,\a_1,\dots,\a_\nu\}$ be a root basis for $R$. From Definition \ref{RootBasis}, for $1\leq j\leq\nu$, we have $$\sg_j=\sum_{i=0}^\nu m_{ji}\a_i,$$ where $m_{ji}\geq0$. For $w\in\w$, let $w(\Pi)=\{w(\a_0),w(\a_1),\dots,w(\a_\nu)\}\subseteq R^\times_b$. Since $w\in\Aut(A)$, $w(\Pi)$ is a $\mathbb{Z}$-basis of $A$.
Also, We have
\begin{equation}\label{wonb}
\sg_j=w(\sg_j)=\sum_{i=0}^\nu m_{ji}w(\a_i)
\end{equation}
Thus $w(\Pi)$ is a root basis for $R$.

\begin{DEF}\label{wconorb}
Let $\Pi$ be a root basis for $R$.

(i) A root basis $\Pi'$ is called a $\w$-conjugate of $\Pi$, if there exists $w\in\w$ for which $\Pi'=w(\Pi)$. Obviously, conjugacy is an equivalence relation on the set of all root bases for $R$.

(ii) The $\w$-orbit of $\Pi$ is the set of all $\w$-conjugates of $\Pi$.
\end{DEF}

In Lemma \ref{ARbasis} it will be shown that a classic root basis of type $A_1$, in the usual sense of Lie theory, is a root basis in our terminology. From \cite[Proposition 5.9]{Ka}, for $\nu=1$, each root basis of $R$ is a $\w$-conjugate of $\Pi_0$ or $-\Pi_0$. It means that any root basis for $R$ is in one of the two $\w$-orbits of $\Pi_0$ and $-\Pi_0$. This is not true for $\nu>1$. In fact when $\nu>1$, the number of $\w$-orbits is not finite. To see this let $N>1$ and consider
\begin{eqnarray*}
\Pi_n=&\{\b_0=\ep+2\sg_1+2\sg_2,\b_1=-\ep-\sg_1-2n\sg_2,\b_2=-\ep-2\sg_1-\sg_2,\\
&\b_3=-\ep-2\sg_1-2\sg_2+\sg_3,\dots,\b_\nu=-\ep-2\sg_1-2\sg_2+\sg_\nu\}.
\end{eqnarray*}
We have $\sg_1=(2n-1)\b_0+\b_1+(2n-2)\b_2$ and $\sg_i=\b_0+\b_i$, for $2\leq i\leq\nu$. It is easy to show that $\Pi_n$ is a $\mathbb{Z}$-linearly independent set. Thus $\Pi_n$ is a root basis for $R$. Since $\sg_1=(2n-1)\b_0+\b_1+(2n-2)\b_2$, using (\ref{wonb}), we conclude that for any $w\in\w$, $\Pi_n\not=w(\Pi_m)$, i.e., $\Pi_m$ and $\Pi_n$ are not $\w$-conjugate, for $m,n>1$ and $m\not=n$.

Here we show that in each orbit if we know the length function with respect to one basis, it is easy to calculate the length function for other root  bases. This is a consequence of Lemma \ref{conlen}.

\begin{lem}\label{wolf}
Let $\Pi$ and $\Pi'$ be two root bases for $R$. If there exists $w_0\in\w$ such that $\Pi'=w_0(\Pi)$, then for $w\in\w$, we have $$\ell_{\Pi'}(w)=\ell_{\Pi}(w_0^{-1}ww_0).$$
\end{lem}

Here we calculate the length function for $\w$ with respect to $w(\Pi_0)$, for $w\in\w$.

\begin{cor}\label{wPLen}
Let $w_0$ be a fixed element of $\w$. For $w\in\w$, we have
$$\ell_{w_0(\Pi_0)}(w)=|\hti(\ep+\tau)|+\sgn(\hti(\ep+\tau))\d_{\vep(w),1},$$
where $\tau=\vep(w_0)(T(w)+(\vep(w)-1)T(w_0))$.
\end{cor}

\begin{proof}
From Lemma \ref{wolf}, we have $\ell_{w_0(\Pi_0)}(w)=\ell_{\Pi_0}(w_0^{-1}ww_0)$. Using Lemma \ref{identities}, we have
$$\vep(w_0^{-1}ww_0)=\vep(w)\andd T(w_0^{-1}ww_0)=\vep(w_0)(T(w)+(1-\vep(w))T(w_0)).$$
Therefore from Theorem \ref{LFT}, for $\tau=\vep(w_0)(T(w)+(1-\vep(w))T(w_0))$, we conclude that
$$\ell_{w_0(\Pi_0)}(w)=|\hti(\ep+\tau)|+\sgn(\hti(\ep+\tau))\d_{\vep(w),1}.$$
\end{proof}

\section{\bf Affine Kac-Moody case}\setcounter{thm}{0}\label{akmc}

In this section, we show that our results about the length function of the Weyl group of a nullity 1 extended affine root system of type $A_1$ coincide with the results about classical length function on the  affine Weyl group, the Weyl group of an affine Kac-Moody root system of type $A_1$. 

In this section, we assume that $\nu=1$. Also, as in section \ref{wlf}, we assume that $A$ is a free abelian group of rank $\nu+1=2$. Thus $\Lam$ is a free abelian group of rank 1. In this case the only semilattice in $\Lam$ is $\Lam$ itself. So, there is only one extended affine root system of type $A_1$ in $A$, namely
$$\rtid=\Lam\cup(\pm\ep+\Lam).$$

By identifying  $\rtid$ with $\rtid\otimes1$ as a subset of $\v=\mathbb{R}\otimes_{\mathbb{Q}}A$, $\rtid$ is the affine Kac-Moody root system of type $A_1^{(1)}$ (see \cite[TABLE Aff1]{Ka} and \cite[Table 1.24]{ABGP}). Thus the $A_1$-type Weyl group of nullity 1, which is the Weyl group of $\rtid$, is the affine Weyl group of type $A_1$.

We have
$$\v=\mathbb{R}\epsilon\oplus\mathbb{R}\sigma_1\quad\mathrm{and}\quad \Lam=\mathbb{Z}\sigma_1$$
and
$$\rtid=\left\{k\epsilon+m\sigma_1\;|\;k\in\{0,\pm1\},\;m\in\mathbb{Z}\right\}.$$
Let $\alpha_0:=\epsilon$ and $\alpha_1:=-\epsilon+\sigma_1$. For $\alpha=k\epsilon+m\sigma_1\in\rtid$, we have $\alpha=(k+m)\alpha_0+m\alpha_1$. Since $k\in\{0,\pm1\}$ and $m\in\mathbb{Z}$, if $m$ is positive (resp. negative) then $k+m$ is non-negative (resp. non-positive). Thus, if $\uplus$ is the disjoint union, we have $\rtid=\rtid^+\uplus\{0\}\uplus\rtid^-$, where
\begin{equation}\label{R+}
\rtid^+=\{m\alpha_0+(k+m)\alpha_1\;|\;m\geq0\;\mathrm{if}\;k=1,\;m>0\;\mathrm{if}\;k\not=1\}
\end{equation}
and
\begin{equation}\label{R-}
\rtid^-=\{m\alpha_0+(k+m)\alpha_1\;|\;m\leq0\;\mathrm{if}\;k=-1,\;m<0\;\mathrm{if}\;k\not=-1\}.
\end{equation}
$\rtid^+$ is called the set of positive roots and $\rtid^-$ is called the set of negative roots of $\rtid$.

\begin{rem}
In the above decomposition of $\rtid$ into the sets of positive and negative roots, we have considered $\rtid$ as an affine Kac-Moody root system and used the classic definition of positive and negative roots. But, using Definition \ref{PNroot} of positive and negative roots for extended affine root systems of type $A_1$, we could get the same decomposition.
\end{rem}

By the proof of \cite[Proposition 6.5]{Ka}, for each $w\in\w$, there exist unique $n\in\mathbb{Z}$ and $s\in\{0,1\}$ where
\begin{equation}\label{affine unique exp}
w=w_\epsilon^s t_1^n,
\end{equation}
where $t_1=w_{\alpha_1}w_{\alpha_0}$. As in section \ref{wlf}, let $\Pi_0=\{\a_0,\a_1\}$. Then, from \cite[Exercise 3.6]{Ka}, $\ell_{\Pi_0}(w)$ for each $w\in\w$, is the number of positive roots which are mapped to negative roots by $w$. Using the above facts one finds a uniform formula for $\ell_{\Pi_0}$ as follows.

First, for $n,m\in\mathbb{Z}$ and $k\in\{0,\pm1\}$, we have
\begin{equation}\label{Wact1}
(w_{\alpha_1}w_{\alpha_0})^n((k+m)\alpha_0+m\alpha_1)=(m-(2n-1)k)\alpha_0+(m-2nk)\alpha_1.
\end{equation}

To show this identity, let $\beta=(k+m)\alpha_0+m\alpha_1$, then $\beta=k\epsilon+m\sigma_1$. By Lemma \ref{action lem}, we have
\begin{eqnarray*}
(w_{\alpha_1}w_{\alpha_0})^n(\beta) & = & (-1)^{2n}k\epsilon+m\sigma_1-2\sum_{i=1}^{|n|}(-(-1)^{2n-2i+1}k\sigma_1+(-1)^{2n-2i}k0)\\
& = & k\epsilon+m\sigma_1-2nk\sigma_1\\
& = & k\epsilon+(m-2nk)\sigma_1\\
& = & (m-(2n-1)k)\alpha_0+(m-2nk)\alpha_1.
\end{eqnarray*}

Next, as an easy consequence of (\ref{Wact1}), for $m,n\in\mathbb{Z}$ and $k\in\{0,\pm1\}$, we have
\begin{equation}\label{Wact2}
w_{\alpha_0}(w_{\alpha_1}w_{\alpha_0})^n((k+m)\alpha_0+m\alpha_1)=(m-(2n+1)k)\alpha_0+(m-2nk)\alpha_1,
\end{equation}

Now, we compute the length of elements of $\w$ with respect to $\Pi_0$. 
For $m\in\mathbb{Z}$ and $k\in\{0,\pm1\}$, Let $\alpha_{m,k}:=(m+k)\alpha_0+m\alpha_1$.

By (\ref{Wact1}), for $n\in\mathbb{Z}$, we have
$$t_1^n(\alpha_{m,k})=(m-2nk)\alpha_0+(m-2nk)\alpha_1=\alpha_{m-2nk,k}.$$
Without loss of generality, we may assume that $n>0$. Let $\alpha_{m,k}\in R^+$. By (\ref{R+}), we have $t_1^n(\alpha_{m,k})\in R^+$, for $k\in\{-1,0\}$. For $k=1$, $t_1^n(\alpha_{m,k})\in R^-$ if and only if $0\leq m\leq 2n-1$. Thus $\ell_{\Pi_0}(t_1^n)=2n$. Since $\ell_{\Pi_0}(w^{-1})=\ell_{\Pi_0}(w)$, for $w\in\w$, we conclude that
\begin{equation}\label{AffineLF1}
\ell_{\Pi_0}(t_1^n)=2|n|,\quad(n\in\mathbb{Z}).
\end{equation}

By (\ref{Wact2}), we have
$$w_\epsilon t_1^n(\alpha_{m,k})=(m-(2n+1)k)\alpha_0+(m-2nk)\alpha_1=\alpha_{m-2nk,-k}.$$

First, let $n\geq0$.
By (\ref{R+}), we have $w_\epsilon t_1^n(\alpha_{m,k})\in R^+$ for $k\in\{-1,0\}$. Also, for $k=1$, $w_\epsilon t_1^n(\alpha_{m,k})\in R^-$ if and only if $0\leq m\leq 2n$. Thus $\ell(w_\epsilon t_1^n)=2n+1$.
Now, let $n<0$. Again by (\ref{R-}), if $k\in\{0,1\}$ then $w_\epsilon t_1^n(\alpha_{m,k})\in R^+$. For $k=-1$, we have $w_\epsilon t_1^n(\alpha_{m,k})\in R^-$ if and only if $m+2n<0$. It means that $\ell(w_\epsilon t_1^n)=2|n|-1$. So, we have
\begin{equation}\label{AffineLF2}
\ell(w_\epsilon t_1^n)=|2n+1|,\quad(n\in\mathbb{Z}).
\end{equation}

By combining (\ref{AffineLF1}) and (\ref{AffineLF2}), we obtain the following unified formula for the length function $\ell_{\Pi_0}$. For $n\in\mathbb{Z}$ and $s\in\{0,1\}$,
\begin{equation}\label{AffineLFcor}
\ell(w_\epsilon^st_1^n)=2|n|+sgn(n)s.
\end{equation}

In Theorem \ref{uniexpspro}, we offered a unique expression for elements of the Weyl group of an affine reflection system of type $A_1$. For $\nu=1$, this unique expression is different from the unique expression in (\ref{affine unique exp}). It also seems that the offered formula for $\ell_{\Pi_0}$ in Theorem \ref{LFT} and (\ref{AffineLFcor}) are different. Here, we convert the two unique expressions to each other and show that the given formulas for $\ell_{\Pi_0}$ are actually the same.
\begin{eqnarray*}
&w_\ep t_1^n=w_{\ep+n\sg_1}\Longrightarrow\ell_{\Pi_0}(w_\ep t_1^n)=|\hti(\ep+n\sg_1)|=2n+1=2n+\sgn(n).\\
&w_\ep t_1^{-n}=w_{\ep-n\sg_1}\Longrightarrow\ell_{\Pi_0}(w_\ep t_1^{-n})=|\hti(\ep-n\sg_1)|=2n-1=2n+\sgn(n).\\
&t_1^n=w_\ep w_{\ep+n\sg_1}\Longrightarrow\ell_{\Pi_0}(t_1^n)=|\hti(\ep+n\sg_1)|-\sgn(\hti(\ep+n\sg_1)\\
&\quad\quad\quad\;\;=2n+1-1=2n.\\
&t_1^{-n}=w_\ep w_{\ep-n\sg_1}\Longrightarrow\ell_{\Pi_0}(t_1^{-n})=|\hti(\ep-n\sg_1)|-\sgn(\hti(\ep+n\sg_1)\\
&\quad\quad\quad\;\;=2n-1+1=2n.
\end{eqnarray*}
Thus Theorem \ref{LFT} coincides with (\ref{AffineLFcor}) for nullity 1.

In Definition \ref{RootBasis}, we defined a root basis for an extended affine root system of type $A_1$. It seems that this definition, for $\nu=1$, is different from the classical definition of a root basis for an affine Kac-Moody root system. We show that these two definitions are the same in this case. First let us recall the classic definition of a root basis.

\begin{DEF}\label{crootb}\cite[5.9]{Ka}
Let $R$ be an affine root system of type $A_1$. A $\mathbb{Z}$-basis $\Pi=\{\b_0,\b_1\}\subseteq R$ for $A$ is called a classic root basis for $R$ if each root $\a\in R$ can be written in the form $\a=\pm(n_0\b_0+n_1\b_1)$, where $n_0,n_1\in\mathbb{Z}_{\geq0}$.
\end{DEF}

It is obvious that a classic root basis for $R$ is a root basis with Definition \ref{RootBasis}. Recall from Section \ref{wlf}, that for $\nu=1$, we have
$$R=\{k\ep+n\sg_1\;|\;k\in\{0,\pm1\},\;n\in\mathbb{Z}\}.$$

\begin{lem}\label{ARbasis}
Any root basis is a classical root basis.
\end{lem}

\begin{proof}
Let $\Pi=\{\b_0,\b_1\}$ be a root basis for $R$. We show that any $\a\in R$ can be written in a form $\a=\pm(m_0\b_0+m_1\b_1)$, where $m_0,m_1\geq0$.

Let $\sg\in R^0\setminus\{0\}$. Then $\sg=n\sg_1$, for $n\in\mathbb{Z}$. Thus each isotropic root in $R$ is strictly positive or strictly negative. Since $\Pi$ is a root basis, we  have $\sg=\pm(m_0\b_0+m_1\b_1)$, for $m_0,m_1\geq0$. 
Since $\sg\not=0$ is isotropic, we have $m_0,m_1>0$. Now we have
$$\pm(m_0\b_0+m_1\b_1)=\sg=p(\sg)=\pm(m_0p(\b_0)+m_1p(\b_1)).$$
Thus $$m_0(\b_0-p(\b_0))+m_1(\b_1-p(\b_1))=0.$$
For $i=0,1$, we have $\b_i-p(\b_i)=\sgn(\b_i)\ep$. Therefore $m_0\sgn(\b_0)+m_1\sgn(\b_1)=0$. Since $\sgn(\b_0),\sgn(\b_1)\in\{\pm1\}$ and $m_0,m_1>0$, we have $\sgn(\b_0)\sgn(\b_1)=-1$ and $m_0=m_1$. Without loss of generality, we may assume that $\sgn(\b_0)=1$ and $\sgn(\b_1)=-1$. Now, for $m\in\mathbb{Z}$, we have
$$m\b_0+m\b_1=mp(\b_0)+mp(\b_1)\in R^0.$$ Thus $\sg_1=\b_0+\b_1$.

With our assumptions, we have $\ep=\b_0-p(\b_0)$. Let $p(\b_0)=m\b_0+m\b_1$, for $m\in\mathbb{Z}$. Then for $k\ep+n\sg_1\in R$, where $k\in\{0,\pm\}$ and $n\in\mathbb{Z}$, we can write
$$k\ep+n\sg_1=(n+k(1-m))\b_0+(n-km)\b_1.$$
If $n-km>0$ then $n-km+k>k$ and if $n-km<0$ then $n-km+k<k$. This completes the proof.
\end{proof}

\centerline{\bf Acknowledgements}
{This research was in part supported by a grant from IPM (No. 90170217).
The authors also would like to thank  the Center of Excellence for Mathematics,
University of Isfahan.}

\end{document}